\documentclass[runningheads]{llncs}
\usepackage[T1]{fontenc}
\usepackage{graphicx}
\usepackage{comment}
\usepackage{proof}
\usepackage{amsfonts,mathrsfs,amssymb,amsmath}
 \usepackage{rotating}
\usepackage{hyperref}
\usepackage{bussproofs}
\usepackage[english]{babel}
\usepackage{tikz}
\usepgflibrary{arrows}
\usetikzlibrary{matrix,arrows,automata,decorations,fit,backgrounds}
\usetikzlibrary{decorations.pathreplacing,shapes.geometric,decorations.text}

\newcommand{\GSD}{\Gamma\Rightarrow\Delta}

\def\to{\supset}

\newcommand{\infrule}[1]{\scriptstyle{#1}}

\newcommand{\context}[2]{{#1}\{{#2} \}}

\def\<{\langle}
\def\>{\rangle}
\def\to{\supset}
\def\To{\Rightarrow}

\def\ex{\exists}
\def\fa{\forall}
\def\E{\mathcal{E}}

\def\F{\mathcal{F}}
\def\M{\mathcal{M}}

\def\W{\mathcal{W}}
\def\W{\mathcal{W}}

\def\D{\mathcal{D}}

\def\R{\mathcal{R}}
\def\Se{\mathcal{S}}

\def\V{\mathcal{V}}

\def\tr{\triangleright}

\def\fa{\forall}
\def\ex{\exists}

\usepackage{color}

\begin{document}
\title{Nested Sequents for Quantified Modal Logics\thanks{Tim S. Lyon was supported by the European Research Council (ERC) Consolidator Grant 771779 (DeciGUT).}}
%
%
\author{Tim S. Lyon\inst{1}\orcidID{0000-0003-3214-0828}
\and
Eugenio Orlandelli\inst{2}\orcidID{0000-0002-4021-8667} 
}
\authorrunning{T.S. Lyon and E. Orlandelli}
%
\institute{
Institute of Artificial Intelligence, TU Dresden, Dresden, Germany
\email{timothy\_stephen.lyon@tu-dresden.de}
\and
Department of the Arts, University of Bologna, Bologna, Italy\\
\email{eugenio.orlandelli@unibo.it}
}
\maketitle              
\begin{abstract}
This paper studies nested sequents for  quantified modal logics. In particular, it considers extensions of the propositional modal logics definable by the axioms \textbf{D}, \textbf{T}, \textbf{B}, \textbf{4}, and \textbf{5} with varying, increasing, decreasing, and constant domains. Each calculus is proved to have good structural properties: weakening and contraction are height-preserving admissible and cut is (syntactically) admissible. Each calculus is shown to be equivalent to the corresponding axiomatic system and, thus, to be sound and complete. Finally, it is argued that the calculi are internal---i.e., each sequent has a formula interpretation---whenever  the existence predicate is expressible in the language.
\end{abstract}
\section{Introduction}
Generalisations of Gentzen-style sequent calculi have proven useful for developing cut-free and analytic proof systems for many  propositional non-classical logics, including modal and intermediate ones. Among these generalisations are {\em display calculi} \cite{B82}, {\em hypersequents} \cite{A96}, {\em labelled calculi} \cite{S94,V00}, and {\em nested sequents} \cite{B92,K94}. They often allow one to give constructive proofs of important meta-theoretical properties such as decidability~\cite{B09}, interpolation~\cite{FK15}, and automatic countermodel extraction~\cite{L21}. These  systems generalise the structural level of Gentzen-style calculi in different ways in order to express wider classes of  logics. In the case of  propositional modal logics they can express the structure of various relational models. In particular, nested sequents encode tree-like relational models and labelled calculi encode graph-like models. 
 In contrast to  other formalisms (e.g. labelled sequents) nested sequents have the advantage of being internal calculi: each nested sequent has a formula interpretation, and thus, such expressions are not a major departure from the modal language.

Things become more difficult when we add the quantifiers. As it is well known \cite{C02,FM98}, in quantified modal logics (QMLs) we have {\em interaction formulas} such as
$$
\textbf{CBF}:=\; \Box \fa xA\to\fa x\Box A
\qquad \text{and}\qquad
\textbf{BF}:=\;\fa x\Box A\to\Box\fa xA
$$
whose validity depends on the interrelations between the domains of quantification ($\D_w$) of the different worlds ($w$) of the model:  \textbf{CBF} is valid only if domains are {\em increasing}---$w\R v$ implies $\D_w\subseteq \D_v$---and \textbf{BF} is valid only if domains are \emph{decreasing}---$w\R v$ implies $\D_w\supseteq \D_v$. Axiomatically, \textbf{CBF} is derivable from the interaction of the axioms/rules for modalities and those for the classical quantifiers, and  \textbf{BF} is independent from them. However, the situation is radically different for sequent calculi than for axiomatic calculi. The problem is that {\bf BF} becomes derivable when we add standard sequent rules for the quantifiers to a calculus having separated left and right rules for the modalities---i.e., it is derivable in all generalisations of Gentzen-style calculi mentioned above. 

To overcome this issue for nested sequents, we employ a formulation technique motivated by labelled sequent calculi. One way of making {\bf CBF} and {\bf BF} independent of the rules for quantifiers within labelled sequent calculi is to extend the language with {\em domain  atoms} of shape $y\in D(w)$ whose intended meaning is that `$y$ belong to the quantificational domain of the label $w$'~\cite{NP11,V00}. In this way, one can restrict the rules for the quantifiers to the terms belonging to the domain of the label under consideration:
$$
\infer[]{y\in D(w),w:\fa xA,\GSD}{w:A(y/x),y\in D(w),w:\fa xA,\GSD}
\quad
\infer[\infrule{z\text{ fresh}}]{\GSD,w:\fa xA}{z\in D(w),\GSD,w:A(z/x)}
$$
As a consequence, {\bf CBF} and {\bf BF} are derivable only if we extend the basic calculus with rules relating the domains of the distinct labels.

In this paper, we study nested sequent calculi for QMLs with varying, increasing, decreasing, and constant domains. Similar to the use of domain atoms in labelled sequents, we will formulate our nested calculi by  extending  the syntax of  sequents with {\em signatures}---i.e., multisets of terms that restrict the applicability of the rules for the quantifiers at that node of the nested sequent---as was done  in \cite{T11} to define hypersequents for G\"odel-Dummett logic with non-constant domains. In particular, we will use the following rules for the universal quantifier:
$$
\infer[\infrule{ L\forall}]{\context{\Se}{X,y;\fa xA,\GSD}}{\context{\Se}{X,y;A(y
/x),\fa xA,\GSD}}\qquad
\infer[\infrule{ R\fa, \text{ $z$ fresh}}]{\context{\Se}{X;\GSD,\fa xA}}{\context{\Se}{X,z;\GSD,A(z/x)}}$$
and will add signature structural rules for increasing, decreasing, and constant domains (Table \ref{rulesQK}).

 As a consequence, we will be able to define nested calculi that are equivalent to the labelled calculi considered in \cite[Ch.~6]{V00} and \cite[Ch.~12.1]{NP11}. We will show that our nested calculi have good structural properties---all rules are height-preserving invertible, weakening and contraction are height-preserving admissible, and cut is syntactically admissible---and that they characterise the quantified extensions of the propositional modal logics in the cube of normal modalities. One advantage of the present approach is that nested sequents with signatures have a formula interpretation given that the language can express the {\em existence predicate} $\E$. In this paper, we will consider a language with identity so that $\E x$ can be expressed as $\ex y(y=x)$ and it need not be taken as an additional primitive symbol; cf.~\cite{C02}. Thus, our calculi utilise (nested) sequents as expressive as the modal language, showing that our calculi are syntactically economical. 

The rest of the paper is organised as follows: \S  \ref{semantics} sketches the  QMLs considered in the paper, and \S \ref{sec:calculi} introduces the nested calculi for these logics. Then, \S \ref{sec:structural} shows that these calculi have good structural properties distinctive of $\mathsf{G3}$-style calculi, including syntactic cut-elimination, and \S \ref{sec:characterisation} shows that each calculus is sound and compete with respect to its intended semantics. Finally, \S \ref{sec:future} presents some future lines of research.

\section{Quantified Modal Logics}\label{semantics}
{\em -Syntax.} 
Let {\em Rel} be a set containing, for each $n\in\mathbb{N}$, an at most countable set of $n$-ary predicates $R_1^n,\,R_2^n,\dots$, and let {\em Var} be a denumerable set of individual variables. The language $\mathcal{L}$ is defined by the following grammar:
\begin{equation}\tag{$\mathcal{L}$}\label{lan}
A\,::=\, R_i^n(x_1,\dots,x_n)\,|\,x_1=x_2\,|\,\bot\,|\, A\to A\,|\, \fa xA\,|\,\Box A
\end{equation}
\noindent where $x,x_1,\dots,x_n\in${\em Var} and $R_i^n\in${\em Rel}. 
An {\em atomic formula} is a formula of the shape $R_i^n(x_1,\dots,x_n)$ or $x_1=x_2$. We use the following metavariables: $x,y,z$ for variables; $P,Q,R$ for atomic formulas; and $A,B,C$ for formulas. An occurrence of a variable $x$ in a formula is {\em free} if it is not in the scope of $\fa x$; otherwise, it is {\em bound}. A {\em sentence} is a formula without free occurrences of variables. The formulas $\neg A$, $A\wedge B$, \mbox{$A\lor B$}, $\ex xA$, and $\Diamond A$ are defined as expected. We follow the usual conventions for parentheses. 
The \emph{weight} of a formula $|A|$ is defined accordingly: $|R_i^n(x_1,\dots,x_n)| = |x=y| = |\bot| = 0$, $|A \to B| = |A| + |B| + 1$, and $|\fa xA| = |\Box A| = |A| + 1$.
We use $A(y/x)$ to denote the formula obtained from $A$ by replacing each free occurrence of $x$ with an occurrence of $y$, possibly renaming bound variables to avoid capture: if $y\not\equiv x$, then  $(\fa y A)(y/x)\equiv \fa z((A(z/y))(y/x))$, where  $z$ is fresh.

\medskip

\noindent {\em -Semantics.}
A {\em frame} is a triple $\F=\< \W,\,\R,\D\>$, where: 
\begin{itemize}
\item$\W$ is a non-empty set of {\em worlds}; 
\item $\R$ is a binary {\em accessibility relation} defined over $\W$; 
\item$\D$ is a function mapping each $w\in \W$ to a possibly empty set of objects $\D_w$ (the  {\em domain} of $w$); we impose that $\D$ is such that $\D_v\neq \varnothing$ for  some $v\in\W$.
\end{itemize}
 We say that $\F$ has:
\begin{enumerate}
\item {\em increasing domains} if for all $ w,v\in\W$, $w\R v$ implies $D_w\subseteq D_v$;
\item {\em decreasing domains} if for all $ w,v\in\W$, $w\R v$ implies $D_w\supseteq D_v$;
\item {\em constant domains} if for all $ w,v\in\W$, $D_w=D_v$;
\item {\em varying domains} if none of the above conditions hold.
\end{enumerate}

A {\em model} $\M$ is a frame together with a valuation function $\V$ such that for each $w\in W$ and each $R^n$ in {\em Rel}, $\V(w,R_n)\subseteq (D_\W)^n$, where $D_\W=\bigcup_{v\in\W} D_v$. An assignment $\sigma$ is a function mapping each variable to an object in $\D_\W$. We let $\sigma^{ x\tr o}$ be the assignment mapping $x$ to $o\in \D_\W$, which behaves like $\sigma$ for all other variables. Observe that variables are {\em rigid designators} in that their value does not change from one world to another.

The notion of {\em satisfaction} of a formula $A$ at a world $w$ of a model $\M$ under an assignment $\sigma$---to be denoted by $\sigma\Vdash_w^\M A$, possibly omitting $\M$---is defined as follows:\smallskip

\begin{tabular}{lll}
$\sigma\Vdash_w^\M R^n(x_1,\dots,x_n)$&\quad iff\quad\phantom{a}&$\langle \sigma(x_1),\dots,\sigma(x_n)\>\in \V(w,R^n)$\\\noalign{\smallskip}
$\sigma\Vdash_w^\M  x=y$&\quad iff\quad\phantom{a}&$\sigma(x)=\sigma(y)$\\\noalign{\smallskip}
$\sigma\not\Vdash_w^\M \bot$\\\noalign{\smallskip}
$\sigma\Vdash_w^\M A\to B$&\quad iff\quad\phantom{a}&$\sigma\not\Vdash_w^\M A$ or $\sigma\Vdash_w^\M B$\\\noalign{\smallskip}
$\sigma\Vdash_w^\M \fa xA$&\quad iff\quad\phantom{a}& for each $o\in \D_w$, $\sigma^{x\tr o}\Vdash_w^\M A$\\\noalign{\smallskip}
$\sigma\Vdash_w^\M \Box A$&\quad iff\quad\phantom{a}& for each $v\in \W$, $w\R v$ implies $\sigma\Vdash_v^\M A$\\\noalign{\smallskip}
\end{tabular}

\noindent  The notions of {\em truth at a world $w$} ($\Vdash_w^\M A$), {\em truth in a model $\M$} ($\Vdash^\M A$), {\em validity in a frame $\F$} ($\F\Vdash A$), and validity in class $\mathcal{C}$ of frames ($\mathcal{C}\Vdash A$) are defined as usual. 
It is well-known that the formula:
\begin{description}
\item[CBF:=] $\Box \fa xA\to\fa x\Box A$ is valid over frames with increasing domains;
\item[BF:=] $\fa x\Box A\to\Box\fa xA$ is valid over frames with decreasing domains;
\item[UI:=] $\fa xA\to A(y/x)$ is valid over frames with constant domains.
\end{description}

\begin{table}[t]
\begin{center}\caption{ Axioms and corresponding properties}\label{corr}\framebox{
\scalebox{0.81}{\begin{tabular}{l|l|l||l|l|l}
Name&Axiom&Property ($w,v,u\in\W)$&Name&Axiom &Property ($w,v,u\in\W)$\\\noalign{\smallskip}\hline\hline\noalign{\smallskip}
{\bf D}& $\Box A\to\Diamond A$& $\fa w\ex u\in\W(w\R u)$&

\,{\bf 5} & $\Diamond A\to\Box\Diamond A$&$\forall w,v,u(w\R v\wedge w\R u\to v\R u)$
\\\noalign{\smallskip}
{\bf T}&$\Box A\to A$&$\forall w(w\R w)$&\,{\bf CBF}&$\Box\fa xA\to\fa x\Box A$\;&$\forall w,v(w\R v\to \D_w\subseteq \D_v)$\\\noalign{\smallskip}
{\bf B}&$A\to\Box\Diamond A$&$\forall w,v(w\R v\to v\R w)$&\,{\bf BF}&$\fa x\Box A\to\Box\fa xA$&$\forall w,v(w\R v\to \D_w\supseteq \D_v)$\\\noalign{\smallskip}
{\bf 4}&$\Box A\to\Box\Box A$&$\forall w,v,u(w\R v \wedge v\R u\to w\R u)$&\,{\bf UI}&$\fa xA\to A[y/x]$&$\fa w,v(\D_w=\D_v)$\\
\end{tabular}}}
\end{center}
\end{table}

\noindent Over frames with non-constant domains the valid theory of quantification is that of positive free logic  instead of that of classical logic. This means that the axiom {\bf UI} is replaced by the weaker axiom {\bf UI}$^\circ:= \fa y(\fa xA\to A(y/x))$. If we extend the language with an \emph{existence predicate} $\E$---whose satisfaction clause is $\sigma\models_w^\M \E x$ iff $\sigma(x)\in\D_w$---then we have the following weaker form of UI that is valid {\bf UI}$^\E:= \fa xA\wedge \E y \to A(y/x)$. Over the language $\mathcal{L}$ the formula $\E x$ can be defined as $\exists y(y=x)$, but over an identity-free language  the existence predicate has to be taken as an additional primitive symbol. This distinction has an impact on the  calculi introduced in the next section: nested sequents have a formula interpretation when $\E$ is expressible in the language.

\medskip

\noindent {\em -Logics.} A {\em QML} is defined to be the set of all formulas that are valid in some given class of frames. 
In this paper, we consider logics that are defined by imposing combinations of the properties in Table \ref{corr}. We use $\mathsf{Q.L}$ for a generic logic  and we say that a formula is {\em $\mathsf{Q.L}$-valid} if it belong to the logic $\mathsf{Q.L}$. The formulas that are valid over the class of all frames is called $\mathsf{Q.K}$ and it is axiomatised by the axioms and rules given in Table \ref{axioms}. We notice that {\bf UI}$^\E$ is a theorem of $\mathsf{Q.K}$, see \cite[Lem.~2.1(iii)]{C02}. The additional axioms for the logics extending $\mathsf{Q.K}$ are given in Table \ref{corr}. We follow the usual conventions for naming logics---e.g., $\mathsf{Q.S4\oplus CBF}$ is the set of formulas that are valid over all reflexive and transitive frames with increasing domains and it is axiomatised by adding axioms \textbf{T}, \textbf{4}, and \textbf{CBF} to $\mathsf{Q.K}$.  We will not distinguish between a logic and its axiomatisation. This is justified by the following theorem.

\begin{theorem}[\cite{C02}]\label{axcomp}
A formula  is a theorem of $\mathsf{Q.L}$ if and only if it is $\mathsf{Q.L}$-valid.
\end{theorem}

\begin{table}[t]
\caption{Axiomatisation of $\mathsf{Q.K}$.}\label{axioms}
\framebox{
\begin{minipage}{0.47\textwidth}
\begin{description}
\item[TAUT.] Propositional tautologies
\item[K.] $\Box (A\to B)\to(\Box A\to \Box B)$
\item[UI$^\circ$.] $\fa y(\fa xA\to A(y/x))$
\item[$\fa$-COMM.] $\fa x\fa yA\to\fa y\fa xA$ 
\item[$\fa$-DIST.] $\fa x(A\to B)\to(\fa xA\to\fa xB)$
\item[$\fa$-VAQ.] $ A\to \fa xA$, if $x$ is not free in $A$
\end{description}
\end{minipage}
\begin{minipage}{0.49\textwidth}
\begin{description}
\item[REF.] $x=x$
\item[REPL.] $x=y\wedge A(x/z)\to A(y/z)$
\item[ND.] $x\neq y\to \Box (x\neq y)$
\item[\phantom{a}]\item[MP.] If $A$ and $A\to B$ are theorem so is $B$
\item[N.] If $A$ is a theorem so is $\Box A$ 
\item[UG.] If $A$ is a theorem so is $\fa x A$\end{description}
\end{minipage}\smallskip
}

\end{table}

\section{Nested Calculi for QML}\label{sec:calculi}

A {\em sequent} is an expression $X;\GSD$ where $X$ is a multiset of variables, called a {\em signature}, and $\Gamma,\,\Delta$ are  multisets of formulas of the language $\mathcal{L}$.  The signature of a sequent is a syntactic counterpart of the existence atoms used in calculi where {\bf UI} is replaced by  {\bf UI}$^\circ$ or {\bf UI}$^\E$, see \cite{MO19}. {\em Nested sequents} are defined as follows:
$$
\Se\;::=\; X;\GSD\;|\; \Se,\,[\Se],\dots,[\Se]
$$ 

\noindent A nested sequent $\Se$ codifies  the  tree of sequents $\mathtt{tr}(\Se)$, as shown in Figure \ref{tree}.
\begin{figure}
\begin{center}\begin{tikzpicture}

\node at (0,-0.1) {$X;\GSD$};
\node at (0,0.5) {$\cdots$};\draw[-latex] (-0.1,0.1) -- (-2,1);
\draw[-latex] (0.1,0.1) -- (2,1);
\node at (-2,1) [above] {$\mathtt{tr}(\Se_1)$};
\node at (2,1) [above] {$\mathtt{tr}(\Se_n)$};
\node at (0,1) [above] {$\cdots$};\end{tikzpicture}\end{center}\caption{The tree of the sequent $X;\GSD,[\Se_1],\dots,[\Se_n]$.}\label{tree}
\end{figure}
 
 \noindent Substitution of free variables are extended to (nested) sequents and to multisets of formulas by applying them component-wise. The formula interpretation of a sequent is defined  as follows:
$$
\texttt{fm}(X;\GSD) \equiv \bigwedge_{x\in X} \E x  \wedge\bigwedge \Gamma\to\bigvee\Delta
$$
where $\E x$ is short for the formula $\ex y(y=x)$ and an empty conjunction (disjunction) is  $\top$ ($\bot$, resp.). To provide a formula reading of nested sequents over the identity-free language we could add $\E$ to the language 
or interpret formulas via their universal closure. In the latter case, for example, the formula interpretation of a sequent would be $\texttt{fm}(X;\GSD)\equiv \forall x\in X(\bigwedge\Gamma\to\bigvee\Delta)$, and it seems our nested calculi would capture the QMLs in \cite{Kripke}.\footnote{We thank the anonymous reviewer who suggested this latter possibility.} Nonetheless, we believe there are independent reasons for studying QMLs over a language containing identity; cf.~\cite{C02,FM98}. The formula interpretation of a nested sequent is defined recursively as:
$$
\texttt{fm}(X;\GSD,[\Se_1],\dots,[\Se_n])\equiv(\bigwedge_{x\in X}\E x\wedge\bigwedge \Gamma\to\bigvee\Delta) \lor \bigvee_{k=1}^{n}\Box \,\texttt{fm}(\Se_k) 
$$

Rules are based on the notion of a {\em hole} $\{\cdot\}$, which is a placeholder for a subtree of (the tree of) a nested sequent and, thus, allows one to apply a rule at an arbitrary node in the tree of a nested sequent. A {\em context} is defined as follows:
$$
\mathcal{C} ::= X;\GSD,\{\cdot\},\dots,\{\cdot\}\ \;|\; \mathcal{C} ,\,[\mathcal{C} ],\dots,[\mathcal{C}]
$$ 
 In other words, a context $\mathcal{C}$ is a nested sequent with $n \geq 0$ hole occurrences, which do not occur inside formulas and must occur within consequent position. We hitherto write contexts as $\Se\{\cdot\}\cdots\{\cdot\}$ indicating each of the holes occurring within the context. The {\em depth} of a hole in a context is defined as the height of the branch from that hole to the root (cf.~\cite{B09}), and we write $Depth(\Se\{\cdot\}) \geq n$ for $n \in \mathbb{N}$ to mean that the depth of the hole in $\mathtt{tr}(\Se\{\cdot\})$ is $n$ or greater.

We define \emph{substitutions} of nested sequents into contexts recursively on the number and depth of holes in a given context: suppose first that our context is of the form $\Se\{\cdot\} \equiv X;\GSD,\{\cdot\},[\Se_{1}], \ldots, [\Se_{n}]$ with a single hole at a depth of $0$ and let $\Se' \equiv Y, \Pi \Rightarrow \Sigma, [\Se_{1}'], \ldots, [\Se_{k}']$ be a nested sequent. Then,
$$
\Se\{\Se'\} \equiv X,Y; \Pi, \GSD, \Sigma, [\Se_{1}], \ldots, [\Se_{n}], [\Se_{1}'], \ldots, [\Se_{k}']
$$
 If our context is of the form $\Se\{\cdot\} \equiv X;\GSD,[\Se_{1}\{\cdot\}], \ldots, [\Se_{n}]$ with a single hole at a depth greater then $0$, then we recursively define $\Se\{\Se'\}$ to be the nested sequent $X;\GSD,[\Se_{1}\{\Se'\}], \ldots, [\Se_{n}]$. This definition extends to a context $\Se\{\cdot\}\cdots\{\cdot\}$ with $n$ holes in the expected way, and for nested sequents $\Se_1, \dots, \Se_n$, we let $\Se\{\Se_1\}\cdots\{\Se_n\}$ denote the nested sequent obtained by replacing, for each $ i\in\{1,\dots,n\}$, the $i$-th hole $\{\cdot\}$ in $\Se\{\cdot\}\cdots\{\cdot\}$ with $\Se_i$. We may also write $\Se\{\Se_{1}\}\{\Se_{i}\}^{n}_{i=2}$ to indicate $\Se\{\Se_{1}\}\cdots\{\Se_{n}\}$ more succinctly. Plugging $\emptyset$ into a hole suggests the removal of the hole; for instance, if $\Se\{\cdot\}\{\cdot\} \equiv x; A \To B, \{\cdot\}, [x,y, B,C \To D, \{\cdot\}]$, then $\Se\{\cdot\}\{\emptyset\} \equiv x; A \To B, \{\cdot\}, [x,y; B,C \To D]$.

The rules of the nested calculi for QMLs are given in Table \ref{rulesQK}. The minimal calculus $\mathsf{NQ.K}$ contains initial sequents, the logical rules, and  the rules for identity (rule {\em Rig} is needed---and is sound---because variables are rigid designators). 
If $\mathsf{Q.L}$ is an extension of $\mathsf{Q.K}$ as discussed in Section \ref{semantics}, then $\mathsf{NQ.L}$ denotes the nested calculus  extending $\mathsf{NQ.K}$ with the rules for the axioms of those logics. Observe that to capture axioms {\bf D}, {\bf CBF}, {\bf BF}, and {\bf UI} we have added structural rules instead of logical ones since the former have a better behaviour.

 In~\cite{B09}, Br\"unnler only considers nested calculi (for propositional modal logics) defined relative to \emph{45-complete sets} of axioms. This restriction is required to ensure that the nested calculi contain all rules required for their completeness. Similarly, in the first-order setting, we only consider nested calculi defined relative to \emph{properly closed} sets of axioms, which is a generalisation of 45-completeness and takes care of the interaction of \textbf{B} with \textbf{CBF} and \textbf{BF} (for example), ensuring the completeness of our nested calculi.

\begin{definition}[Properly Closed] Let $\mathsf{L} \subseteq \{\mathbf{D}, \mathbf{T}, \mathbf{B}, \mathbf{4}, \mathbf{5}, \mathbf{CBF}, \mathbf{BF}, \mathbf{UI}\}$. We define $\mathsf{L}$ to be \emph{properly closed} iff  if all $\mathsf{Q.L}$-frames satisfy $X\in \{\mathbf{4}, \mathbf{5}, \mathbf{CBF}, \mathbf{BF}\}$, then $X \in \mathsf{L}$. We define a nested calculus $\mathsf{NQ.L}$ to be \emph{properly closed} iff (1) $\mathsf{L}$ is properly closed, and (2) $R_{5dom} \in \mathsf{NQ.L}$ iff $\mathbf{5} \in \mathsf{L}$ and $\{\mathbf{CBF},\mathbf{BF}\} \cap \mathsf{L} \neq \emptyset$. 
\end{definition}

\begin{remark}
All nested calculi hitherto considered will be assumed properly closed.
\end{remark}

\begin{table}[t]\begin{center}\caption{Nested rules for QML}\label{rulesQK}
\scalebox{0.79
}{\begin{tabular}{ccc}
\hline\hline\noalign{\medskip}
{\bf Initial Sequents:}& \multicolumn{2}{l}{$\context{\Se}{X;P,\Gamma\To\Delta,P}$ with $P$ atomic}\\\noalign{\medskip}
{\bf Logical Rules:}&\\\noalign{\medskip}
\infer[\infrule L\to]{\context{\mathcal{S}}{X;A\to B,\GSD}}{\context{\mathcal{S}}{X;\GSD,A}&\context{\mathcal{S}}{ X;B,\GSD}}
&
\infer[\infrule R\to]{\context{\mathcal{S}}{X;\GSD,A\to B}}{\context{\mathcal{S}}{X;A,\GSD,B}}
&
\infer[\infrule L\bot]{\mathcal{S}\{X;\bot,\GSD\}}{}\\\noalign{\medskip}
\infer[\infrule{ L\forall}]{\context{\Se}{X,z;\fa xA,\GSD}}{\context{\Se}{X,z;A(z/x),\fa xA,\GSD}}&
\infer[\infrule{ R\fa, \text{ $y$ fresh}}]{\context{\Se}{X;\GSD,\fa xA}}{\context{\Se}{X,y;\GSD,A(y/x)}}\\\noalign{\medskip}
\infer[\infrule L\Box]{\Se\{X;\Box A,\GSD,[Y;\Pi\To\Sigma]\}}{\Se\{X;\Box A,\GSD,[Y;A,\Pi\To\Sigma]\}}&
\multicolumn{1}{c}{\infer[\infrule R\Box]{\Se\{X;\GSD,\Box A\}}{\Se\{X;\GSD,[\emptyset; \, \To A]\}}}\\\noalign{\medskip}
{\bf Identity Rules:}\\\noalign{\medskip}
\infer[\infrule Ref]{\Se\{X;\GSD\}}{\Se\{X;x=x,\GSD\}}&
\multicolumn{2}{c}{\infer[\infrule Repl]{\Se\{X;x=y,P(x/z),\GSD\}}{\Se\{X;P(y/z),x=y,P(x/z),\GSD\}}}
\\\noalign{\medskip}
\infer[\infrule Repl_X]{\Se\{X,x;x=y,\GSD\}}{\Se\{X,x,y;x=y,\GSD\}}
&
\multicolumn{2}{c}{\infer[\infrule Rig]{\Se\{X;x=y,\GSD\}\{Y;\Pi\To\Sigma\}}{\Se\{X;x=y,\GSD\}\{Y;x=y,\Pi\To\Sigma\}}}
\\\noalign{\medskip}\hline\noalign{\smallskip}
\multicolumn{1}{l}{{\bf Rules for Propositional Axioms:} }&
\\\noalign{\medskip}
\infer[\infrule R_D]{\context{\Se}{X;\GSD}}{\context{\Se}{X;\GSD,[\emptyset; \,\To\,]}}&

\infer[\infrule R_B\;\;]{\context{\Se}{X;\GSD,[Y;\Box A,\Pi\To\Sigma]}}{\context{\Se}{X;A,\GSD,[Y;\Box A,\Pi\To\Sigma]}}
&
\infer[\infrule R_T]{\context{\Se}{X;\Box A,\GSD}}{\context{\Se}{X;A,\Box A,\GSD}}\\\noalign{\medskip}
\infer[\infrule R_4]{\context{\Se}{X;\Box A,\GSD,[Y;\Pi\To\Sigma]}}{\context{\Se}{X;\Box A,\GSD,[Y;\Box A,\Pi\To\Sigma]}}
&
\multicolumn{2}{c}{\infer[\infrule R_5,\; Depth(\Se\{\cdot\}\{\emptyset\}) \geq 1]{\context{\Se}{X; \Box A,\GSD}\{Y;\Pi\To\Sigma\}}{\context{\Se}{X;\Box A,\GSD}\{Y;\Box A,\Pi\To\Sigma\}}}
\\\noalign{\medskip}
\multicolumn{1}{l}{{\bf Rules for Domains:} }
\\\noalign{\medskip}
\infer[\infrule R_{cbf}]{\context{\Se}{X,x;\GSD,[Y;\Pi\To\Sigma]}}{\context{\Se}{X,x;\GSD,[Y,x;\Pi\To\Sigma]}}
&
\infer[\infrule R_{bf}]{\context{\Se}{X;\GSD,[Y,x;\Pi\To\Sigma]}}{\context{\Se}{X,x;\GSD,[Y,x;\Pi\To\Sigma]}}
&
\infer[\infrule R_{ui}]{\context{\Se}{X;\GSD}}{\context{\Se}{X,x;\GSD}}
\\\noalign{\medskip}
\multicolumn{3}{c}{\infer[\infrule R_{5dom},\; Depth(\Se\{\emptyset\}\{\cdot\}) \geq 1 \text{ and } Depth(\Se\{\cdot\}\{\emptyset\}) \geq 1]{\context{\Se}{X,x; \GSD}\{Y;\Pi\To\Sigma\}}{\context{\Se}{X,x;\GSD}\{Y,x;\Pi\To\Sigma\}}}

\\\noalign{\medskip}
\\\noalign{\medskip}\hline\hline
\end{tabular}}\end{center}
\end{table}

%

Given a calculus $\mathsf{NQ.L}$, an \emph{$\mathsf{NQ.L}$-derivation} of a nested sequent ${}\Se$ is a tree of nested sequents, whose leaves are initial sequents, whose root is ${}\Se$, and which grows according to the rules of $\mathsf{NQ.L}$. We consider only derivations of \emph{pure sequents}, meaning no variable has both free and bound occurrences and each \emph{eigenvariable} (i.e., a fresh variable participating in an $R\fa$ inference) is distinct. The \emph{height} of an $\mathsf{NQ.L}$-derivation  is the number of nodes of one of its longest branches. We say that $\Se$ is $\mathsf{NQ.L}$-derivable 
if there is an $\mathsf{NQ.L}$-derivation 
of $\Se$ or of an alphabetical variant of $\Se$. We let $\mathsf{NQ.L}\vdash \Se$ denote that $\Se$ is $\mathsf{NQ.L}$-derivable. A rule is said to be \emph{(height-preserving) admissible} in $\mathsf{NQ.L}$, if, whenever its premisses are $\mathsf{NQ.L}$-derivable (with height at most $n$), also its conclusion is {$\mathsf{NQ.L}$-derivable (with height at most $n$). A rule is said to be \emph{(height-preserving) invertible} in $\mathsf{NQ.L}$, if, whenever its conclusion is $\mathsf{NQ.L}$-derivable (with height at most $n$), each premiss is {$\mathsf{NQ.L}$-derivable (with height at most $n$). For each rule displayed in Table \ref{rulesQK}, the formulas explicitly displayed in the conclusion are called \emph{principal}, those explicitly displayed in the premisses are called {\em auxiliary}, and everything else constitutes the \emph{context}. 

\section{Properties and Cut-Elimination}\label{sec:structural}

We now show that our nested calculi satisfy fundamental admissibility and invertibility properties. Ultimately, we will apply these properties in our proof of syntactic cut-elimination.

\begin{lemma}[Generalised Initial Sequents]\label{genax}  $\mathsf{NQ.L}\vdash \Se\{X;A,\GSD,A\}$, for any arbitrary $\mathcal{L}$-formula $A$.
\end{lemma}

\begin{proof} By a standard induction on the weight of $A$.
\qed\end{proof}

\begin{lemma}\label{axid}
The  sequents $
\Se\{\;\To x=x\}$ and $\Se\{x=y,A(x/z)\To A(y/z)\}$ are  $\mathsf{NQ.L}$-derivable. 
\qed \end{lemma}

\begin{proof}
$\Se\{\;\To x=x\}$ is derivable by applying an instance of rule {\em Ref} to the initial sequent $\Se\{\;x=x\To x=x\}$. The case of $\Se\{x=y,A(x/z)\To A(y/z)\}$ is handled by induction on $|A(x/z)|$. We consider only the case where $A(x/z)= \Box B(x/z)$.
$$
\infer[\infrule R\Box]{\Se\{x=y,\Box B(x/z)\To\Box B(y/z)\}}{
\infer[\infrule L\Box]{\Se\{x=y,\Box B(x/z)\To\;,[\,\To B(y/z)]\}}{
\infer[\infrule Rig]{\Se\{x=y,\Box B(x/z)\To\;,[B(x/z)\To B(y/z)\}]}{
\infer[\infrule IH]{\Se\{x=y,\Box B(x/z)\To\;,[x=y,B(x/z)\To B(y/z)]\}}{}}}}
$$
\phantom{a}\qed\end{proof}

\begin{lemma}\label{botr} The following $R\bot$ rule is height-preserving admissible in $\mathsf{NQ.L}$:
$$
\infer[\infrule R\bot ]{\Se\{X;\GSD\}}{\Se\{X;\GSD,\bot\}}
$$
\end{lemma}

\begin{proof} By a straightforward induction on the height of the derivation $\mathcal{D}$ of the premiss. The proof is almost trivial as any application of $R\bot$ to an initial sequent of an instance of $L\bot$ gives another initial sequent or instance of $L\bot$, respectively, and $R\bot$ permutes above every other rule of $\mathsf{NQ.L}$.\qed\end{proof}\begin{lemma}[Substitution]\label{subs} The following rule of substitution of free variables is height-preserving admissible in $\mathsf{NQ.L}$:
$$
\infer[\infrule (y/x) ]{\Se(y/x)\{X(y/x);\Gamma(y/x)\To\Delta(y/x)\}}{\Se\{X;\GSD\}}
$$
\end{lemma}

\begin{proof}
By induction on the height of the derivation $\mathcal{D}$ of the premiss. The only interesting case is when the last step of $\mathcal{D}$ is an instance of $R\fa$:
$$
\infer[\infrule{ R\fa, \text{ $z_2$ fresh}}]{\context{\Se}{X;\GSD,\fa z_1A}}{\context{\Se}{X,z_{2};\GSD,A(z_2/z_1)}}
$$
We transform the derivation of the premiss by applying the inductive hypothesis twice to ensure the freshness condition is preserved: the first time to replace $z_2$ with a fresh variable $z_3$ and then to replace $x$ with $y$. We conclude by applying $R\fa$ with $z_3$ as the {\em eigenvariable}.
\qed\end{proof}

Typically, admissible structural rules operate on either formulas (e.g., see the internal weakening rule IW below) or nesting structure (e.g., see the Merge rule below) in nested calculi. An interesting observation in the first-order setting is that admissible structural rules also act on the signatures occurring in nested sequents. This gives rise to forms of weakening and contraction for terms, which are reminiscent of analogous rules formulated in the context of hypersequents with signatures~\cite{T11}.

\begin{lemma}[Signature Structural Rules]\label{sig-struc-rules} The following rules of signature weakening and signature contraction are height-preserving admissible in $\mathsf{NQ.L}$:
$$\infer[\infrule SW]{\context{\Se}{X,x;\GSD}}{\context{\Se}{X;\GSD}}\qquad
\infer[\infrule SC]{\context{\Se}{X,x;\GSD}}{\context{\Se}{X,x,x;\GSD}}$$
\end{lemma}

\begin{proof} By a standard induction on the height of the derivation $\mathcal{D}$ of the premiss. Proving height-preserving admissibility of SC is trivial as the rule permutes above all rules of $\mathsf{NQ.L}$. Proving the height-preserving admissibility of SW is also straightforward with the only interesting case arising when $\mathcal{D}$ ends with an instance of $R\fa$ with $x$ as the {\em eigenvariable}. However, this  case is easily managed by applying the height-preserving admissible substitution  $(y/x)$ to ensure the freshness condition for $R\fa$ is satisfied, followed by the inductive hypothesis, and an application of $R\fa$.
\qed\end{proof}

As in the setting of first-order intuitionistic logics with increasing and constant domains (see~\cite{L20}), we find that our structural rules for domains give rise to admissible logical rules generalising the $L\fa$ rule. Such rules (presented in the proposition below) combine the functionality of the associated domain structural rules with the $L\fa$ rule. The $L\fa_{bf}$ and $L\fa_{cbf}$ rules are instances of \emph{reachability rules}~\cite{L21,L22}, which bottom-up operate by searching for terms along edges in a nested sequent used to instantiate universal formulas.

\begin{proposition} The following logical rules for `domain-axioms' and for axiom {\bf D} are admissible in the nested calculi including the appropriate structural rules for domains or $R_D$:
\begin{center}
\scalebox{0.9}{\begin{tabular}{cc c}
\infer[\infrule L\fa_{bf}]{\context{\Se}{X;\fa xA,\GSD,[Y,y;\Pi\To\Sigma]}}{\context{\Se}{X;A(y/x),\fa xA,\GSD,[Y,y;\Pi\To\Sigma]}}

&\quad&

\infer[\infrule L\fa_{ui}]{\context{\Se}{X;\fa xA,\GSD}}{\context{\Se}{X;A(y/x),\fa xA,\GSD}}
\end{tabular}}
\end{center}
\[
\scalebox{0.9}{\infer[\infrule L\fa_{cbf}]{\context{\Se}{X,y;\GSD,[Y;\fa xA,\Pi\To\Sigma]}}{\context{\Se}{X,y;\GSD,[Y;A(y/x),\fa xA,\Pi\To\Sigma]}}\quad
\infer[\infrule L_D]{\Se\{X;\Box A,\GSD\}}{\Se\{X;\Box A,\GSD,[\emptyset;A\To \;]\}}
}
\]
\end{proposition}

\begin{proof} The admissibility of $L\fa_{cbf}$ from $R_{cbf}$ and $SW$ is proven as follows:
$$
\infer[\infrule R_{cbf}]{\Se\{X,y;\GSD,[Y;\fa xA,\Pi\To\Sigma]\}}{
\infer[\infrule L\fa]{\Se\{X,y;\GSD,[Y,y;\fa xA,\Pi\To\Sigma]\}}{
\infer[\infrule SW]{\Se\{X,y;\GSD,[Y,y; A(y/x),\fa xA,\Pi\To\Sigma]\}}{
\Se\{X,y;\GSD,[Y;A(y/x),\fa xA,\Pi\To\Sigma]\}}}}
$$
The cases of $L\fa_{bf}$ and $L\fa_{ui}$ are similar, and the case of $L_D$ follows immediately from $R_D$.
\qed\end{proof}
\begin{lemma}[Weakenings]\label{weakening} The following rules of internal and external weakening are height-preserving admissible in $\mathsf{NQ.L}$:
$$\infer[\infrule IW]{\context{\Se}{X; \Pi,\GSD,\Sigma}}{\context{\Se}{X; \GSD}}\qquad
\infer[\infrule EW]{\context{\Se}{X; \GSD,[Y; \Pi\To\Sigma]}}{\context{\Se}{X; \GSD}}$$
\end{lemma}

\begin{proof}
By induction on the height of the derivation $\mathcal{D}$ of the premiss. If $\mathcal{D}$ ends with an instance of rule $R\fa$ with $y$ the {\em eigenvariable}, we apply the (height-preserving admissible) substitution rule to replace $y$ with a fresh variable $z$ occurring neither in $\context{\Se}{X; \GSD}$, nor in $\Pi,\Sigma$ (in the IW case) or in $Y,\Pi,\Sigma$ (in the EW case). Then, we apply the inductive hypothesis and an instance of $R\fa$ to conclude $\context{\Se}{X; \Pi,\GSD,\Sigma}$ in the IW case and $\context{\Se}{X; \GSD,[Y; \Pi\To\Sigma]}$ in the EW case.\qed\end{proof}

\begin{lemma}[Necessitation and Merge]\label{merge} The following rules are height-preserving admissible in $\mathsf{N.QL}$:
$$
\infer[\infrule Nec]{\To\,,[\Se]}{\Se}
\qquad
\infer[\infrule Merge]{\Se\{X;\GSD,[Y,Z;\Pi_1,\Pi_2\To\Delta_1,\Delta_2]\}}{\Se\{X;\GSD,[Y;\Pi_1\To\Delta_1],[Z;\Pi_2\To\Delta_2]\}}
$$
\end{lemma}

\begin{proof} By a simple induction on the height of the derivation 
 of the premiss.
\qed\end{proof}

\begin{lemma}[Invertibility]\label{lem:invert} Each rule of $\mathsf{NQ.L}$ is height-preserving invertible.
\end{lemma}
\begin{proof} The proof is by induction on the height of the derivation. The height-preserving invertibility of all rules but $L{\to}$, $R{\to}$, $R\fa$ and $R\Box$ follows from Lemmas \ref{sig-struc-rules} and \ref{weakening}, and the proof of the remaining cases is standard.\qed\end{proof}

\begin{lemma}[Contraction]\label{contr} The following rules of left and right contraction are height-preserving admissible in $\mathsf{NQ.L}$:
\begin{center}
\begin{tabular}{c @{\hskip 2em} c}
\infer[\infrule CL]{\context{\mathcal{S}}{X;\Gamma,A\To\Delta}}{\context{\mathcal{S}}{X;\Gamma,A,A\To\Delta}}

&

\infer[\infrule CR]{\context{\mathcal{S}}{X;\Gamma\To\Delta,A}}{\context{\mathcal{S}}{X;\Gamma\To\Delta,A,A}}
\end{tabular}
\end{center}
\end{lemma}

\begin{proof} By simultaneous induction on the height of the derivation of the premisses of CL and  CR. We consider only the non-trivial $R \forall$ case for CR as the remaining cases are similar or simpler. Assume that the last step of $\mathcal{D}$ is: 
$$
\infer[\infrule R\fa]{\context{\Se}{X;\GSD,\fa xA, \fa xA}}{\context{\Se}{X,y;\GSD,A(y/x),\fa xA}}
$$
 To resolve the case, we apply the height-preserving invertibility of $R\fa$, the height-preserving admissibility of $(y/z)$ and SC, followed by the inductive hypothesis. Finally, an application of $R\fa$ gives the desired conclusion.
$$
\infer[\infrule R\fa]{\context{\Se}{X;\GSD,\fa xA}}{
\infer[\infrule IH]{\context{\Se}{X,y;\GSD,A(y/x)}}{
\infer[\infrule SC]{\context{\Se}{X,y;\GSD,A(y/x),A(y/x)}}{
\infer[\infrule (y/z)]{\context{\Se}{X,y,y;\GSD,A(y/x),A(y/x)}}{
\infer[\infrule Lemma~\ref{lem:invert}]{\context{\Se}{X,y,z;\GSD,A(y/x),A(z/x)}}{
{\context{\Se}{X,y;\GSD,A(y/x),\fa xA}}
}}}}}
$$
\qed\end{proof}

Due to the presence of $R_4$ and $R_5$ in specific nested calculi, our cut elimination theorem (Theorem~\ref{cut} below) requires us to simultaneously eliminate a second form of cut that acts on modal formulas. We refer to this rule as $\mathsf{L}$-Cut and note that it is essentially Br\"unnler's $\mathsf{Y}$-$\mathsf{cut}$ rule~\cite{B09}. Since the principal and auxiliary formulas of $R_4$ and $R_5$ are of the same weight (i.e. both are $\Box A$), $\mathsf{L}$-Cut is needed to permute the cut upward in these special cases as cuts cannot be reduced to formulas of a smaller weight.

\begin{definition}[$\mathsf{L}$-Cut and $\mathsf{L}$-Str] Let $\mathsf{NQ.L}$ be properly closed. We define $\mathsf{L}$-Cut to be the following rule:
\begin{center}
\resizebox{\columnwidth}{!}{\AxiomC{$\Se\{X;\GSD,\Box A\}\{Y_{i}; \Pi_{i} \To \Sigma_{i}\}^{n}_{i=1}$}
\AxiomC{$\Se\{X;\Box A,\GSD\}\{Y_{i}; \Box A, \Pi_{i} \To \Sigma_{i}\}^{n}_{i=1}$}
\RightLabel{$\infrule{\mathsf{L}\text{-}Cut}$}
\BinaryInfC{$\Se\{X;\GSD\}\{Y_{i}; \Pi_{i} \To \Sigma_{i}\}^{n}_{i=1}$}
\DisplayProof}
\end{center}
which is subject to the following side conditions:
\begin{itemize}

\item if $\mathbf{4}, \mathbf{5} \not\in \mathsf{L}$, then $n = 0$;

\item if $\mathbf{4} \in \mathsf{L}$ and $\mathbf{5} \not\in \mathsf{L}$, then $\Se\{\cdot\}\{\cdot\}$ is of the form $\Se\{X; \GSD, \{\cdot\},\{\Se_{1}\{\cdot\}^{n}\}\}$;

\item if $\mathbf{5} \in \mathsf{L}$ and $\mathbf{4} \not\in \mathsf{L}$, then $Depth(\Se\{\cdot\}\{\emptyset\}^{n}) \geq 1$;

\item otherwise, if $\mathbf{4}, \mathbf{5} \in \mathsf{L}$, then no restriction on the shape of the rule is enforced.

\end{itemize}
We define $\mathsf{L}$-Str to be the following rule:
\begin{center}
\AxiomC{ }
\noLine
\UnaryInfC{$\Se\{Y_{1}; \Pi_{1} \Rightarrow \Sigma_{1}, [X;\GSD]\}\{Y_{2}; \Pi_{2} \Rightarrow \Sigma_{2}\}$}
\RightLabel{$\mathsf{L}$-Str}
\UnaryInfC{$\Se\{Y_{1}; \Pi_{1} \Rightarrow \Sigma_{1}\}\{Y_{2}; \Pi_{2} \Rightarrow \Sigma_{2}, [X;\GSD]\}$}
\DisplayProof
\end{center}
which is subject to the following side conditions:
\begin{itemize}

\item if $\mathbf{4}, \mathbf{5} \not\in \mathsf{L}$, then $\Se\{\cdot\}\{\cdot\}$ is of the form $\Se\{X; \GSD, \{\cdot\},\{\cdot\}\}$;

\item if $\mathbf{4} \in \mathsf{L}$ and $\mathbf{5} \not\in \mathsf{L}$, then $\Se\{\cdot\}\{\cdot\}$ is of the form $\Se\{X; \GSD, \{\cdot\},\{\Se_{1}\{\cdot\}\}\}$;

\item if $\mathbf{5} \in \mathsf{L}$ and $\mathbf{4} \not\in \mathsf{L}$, then $Depth(\Se\{\cdot\}\{\emptyset\}) \geq 1$;

\item otherwise, if $\mathbf{4}, \mathbf{5} \in \mathsf{L}$, then no restriction on the shape of the rule is enforced.

\end{itemize}
\end{definition}

\begin{lemma}[Special Structural Rules]\label{specialstruct} If $\mathsf{NQ.L}$ contains the rule $R_X$ for the propositional axiom $X$, then the corresponding structural rule from Table \ref{special} is admissible in $\mathsf{NQ.L}$. Moreover, $\mathsf{L}$-Str is admissible in $\mathsf{NQ.L}$.
\end{lemma}

\begin{table}[t]\begin{center}\caption{Structural rules for propositional axioms}\label{special}
\scalebox{1
}{\begin{tabular}{ccccc}
\hline\hline\noalign{\medskip}
\multicolumn{2}{c}{\infer[\infrule S_T\qquad]{\Se\{X,Y;\Pi,\GSD,\Sigma\}}{\Se\{X;\GSD,[Y;\Pi\To\Sigma]\}}}
&
\multicolumn{3}{c}{\infer[\infrule S_4]{\Se\{X;\GSD,[\emptyset; \,\To\,,[Y;\Pi\To\Sigma]\,]\}}{\Se\{X;\GSD,[Y;\Pi\To\Sigma]\}} }
\\\noalign{\medskip}
\multicolumn{5}{c}{\infer[\infrule S_5,\;Depth(\Se\{\cdot\}\{\emptyset\}) \geq 1]{\Se\{Y_{1}; \Pi_{1} \Rightarrow \Sigma_{1}\}\{Y_{2}; \Pi_{2} \Rightarrow \Sigma_{2},[X;\GSD]\}}{\Se\{Y_{1}; \Pi_{1} \Rightarrow \Sigma_{1},[X;\GSD]\}\{Y_{2}; \Pi_{2} \Rightarrow \Sigma_{2}\}} }
\\\noalign{\medskip}
\multicolumn{5}{c}{\infer[\infrule S_B]{\Se\{ X,Z;\Pi_1,\GSD,\Sigma_1,[Y;\Pi_2\To\Sigma_2]\}}{\Se\{ X;\GSD,[Y;\Pi_2\To\Sigma_2,[Z;\Pi_1\To\Sigma_1 ]\,]\}} }
\\\noalign{\medskip}\hline\hline
\end{tabular}}\end{center}
\end{table}

\begin{proof} We argue the $S_B$ case by induction on the height of the given derivation; the remaining cases follow from the lemmas in the appendix. We only consider the $R_{bf}$ and $R_{5dom}$ cases of the inductive step as the remaining cases are simple or similar.
\begin{center}
\AxiomC{$\context{\Se}{Z; \Pi_{1} \To \Sigma_{1},[X,x;\GSD,[Y,x;\Pi_{2}\To\Sigma_{2}]]}$}
\RightLabel{$\infrule{R_{bf}}$}
\UnaryInfC{$\context{\Se}{Z; \Pi_{1} \To \Sigma_{1},[X;\GSD,[Y,x;\Pi_{2}\To\Sigma_{2}]]}$}
\RightLabel{$\infrule{S_B}$}
\UnaryInfC{$\context{\Se}{Z,Y,x; \Pi_{1},\Pi_{2} \To \Sigma_{1},\Sigma_{2},[X;\GSD]}$}
\DisplayProof
\end{center}
 As our nested calculi are assumed to be properly closed, we know that if $\mathsf{NQ.L}$ contains $R_B$ and $R_{bf}$, then it must contain $R_{cbf}$, showing that we can apply $IH$ first and then $R_{cbf}$ as shown below.
\begin{center}
\AxiomC{$\context{\Se}{Z; \Pi_{1} \To \Sigma_{1},[X,x;\GSD,[Y,x;\Pi_{2}\To\Sigma_{2}]]}$}
\RightLabel{$\infrule{IH}$}
\UnaryInfC{$\context{\Se}{Z,Y,x; \Pi_{1},\Pi_{2} \To \Sigma_{1},\Sigma_{2},[X,x;\GSD]}$}
\RightLabel{$\infrule{R_{cbf}}$}
\UnaryInfC{$\context{\Se}{Z,Y,x; \Pi_{1},\Pi_{2} \To \Sigma_{1},\Sigma_{2},[X;\GSD]}$}
\DisplayProof
\end{center}
Last, we consider an interesting $R_{5dom}$ case:
\begin{center}
\AxiomC{$Z;\Pi_{1}\To\Sigma_{1},[X_{1};\Gamma_{1} \To \Delta_{1}, [X_{2},x;  \Gamma_{2} \To \Delta_{2}]],[\Se\{Y,x;\Pi_{2}\To\Sigma_{2}\}]$}
\RightLabel{$\infrule R_{5dom}$}
\UnaryInfC{$Z;\Pi_{1}\To\Sigma_{1},[X_{1};\Gamma_{1} \To \Delta_{1}, [X_{2},x;  \Gamma_{2} \To \Delta_{2}]],[\Se\{Y;\Pi_{2}\To\Sigma_{2}\}]$}
\RightLabel{$\infrule S_{B}$}
\UnaryInfC{$Z,X_{2},x;\Pi_{1},\Gamma_{2}\To\Sigma_{1},\Delta_{2},[X_{1},\Gamma_{1} \To \Delta_{1}],[\Se\{Y;\Pi_{2}\To\Sigma_{2}\}]$}
\DisplayProof
\end{center}
 To resolve the case, we apply the inductive hypothesis, followed by the height-preserving admissible rule SW. We apply the SW rule $n-1$ times adding the variable $x$ along the path from the root to $Y,x;\Pi_{2}\To\Sigma_{2}$, and then the $R_{cbf}$ rule $n$ times to delete the $n-1$ copies of $x$ up to the root. We may apply $R_{cbf}$ as our nested calculi are properly closed, that is,  ${\bf B}, {\bf BF} \in \mathsf{L}$ only if ${\bf CBF} \in \mathsf{L}$.
\begin{center}
\AxiomC{$Z;\Pi_{1}\To\Sigma_{1},[X;\Gamma_{1} \To \Delta_{1}, [X_{2},x; \Gamma_{2} \To \Delta_{2}]],[\Se\{Y,x;\Pi_{2}\To\Sigma_{2}\}]$}
\RightLabel{$\infrule IH$}
\UnaryInfC{$Z,X_{2},x;\Pi_{1},\Gamma_{2}\To\Sigma_{1},\Delta_{2},[X,\Gamma_{1} \To \Delta_{1}],[\Se\{Y,x;\Pi_{2}\To\Sigma_{2}\}]$}
\RightLabel{$\infrule{ SW\, \text{($n-1$ times)}}$}
\UnaryInfC{$Z,X_{2},x;\Pi_{1},\Gamma_{2}\To\Sigma_{1},\Delta_{2},[X,\Gamma_{1} \To \Delta_{1}],[\Se\{Y,x;\Pi_{2}\To\Sigma_{2}\}]$}
\RightLabel{$\infrule{ R_{cbf}\, \text{ ($n$ times)}}$}
\UnaryInfC{$Z,X_{2},x;\Pi_{1},\Gamma_{2}\To\Sigma_{1},\Delta_{2},[X,\Gamma_{1} \To \Delta_{1}],[\Se\{Y;\Pi_{2}\To\Sigma_{2}\}]$}
\DisplayProof
\end{center}
\qed\end{proof}

In our cut-elimination theorem below, we provide a procedure to eliminate an additive (i.e. context-sharing) version of cut as in the work on nested sequents for propositional modal logics by Br\"unnler \cite{B09}. We note that we could have considered an equivalent, multiplicative (i.e. context-independent) version---like the cut rule shown eliminable in the tree-hypersequent systems of Poggiolesi~\cite{P09}---however, we find the additive version of the rule to be simpler as we can forgo considerations of how to fuse nested sequents of a different form.\footnote{Nested sequents and tree-hypersequents are equivalent formalisms; cf.~\cite{B09,P09}.}

\begin{theorem}[Cut]\label{cut} $\mathsf{L}$-Cut and the following rule of Cut are admissible in $\mathsf{NQ.L}$:
$$
\infer[\infrule Cut ]{\Se\{X;\GSD\}}{\Se\{X;\GSD,A\}&\Se\{X;A,\GSD\}}
$$
\end{theorem}

\begin{proof} We consider an uppermost instance of $\mathsf{L}$-Cut or {\em Cut} with $A \equiv \Box B$ and $A$ the cut formula of each rule, respectively. We argue by simultaneous induction on the lexicographic ordering of pairs $(|A|,h_{1}+h_{2})$, where $|A|$ is the weight of $A$ and $h_{1}$ ($h_2$) is the height of the derivation $\mathcal{D}_1$ ($\mathcal{D}_2$) of the left (right) premiss of the  instance of $\mathsf{L}$-Cut or {\em Cut} under consideration.

Let us first consider the case where the weight of $A$ is zero, i.e. $A$ is a formula of the form $R_i^n(x_1,\dots,x_n)$, $\bot$, or $x = y$. The first two cases are standard, so we consider the case when $A$ is of the form $x = y$. We suppose first that $x=y$ is not principal in the left premiss of {\em Cut}. If the left premiss is an initial sequent or an instance of $L\bot$, then the conclusion  will be as well, so we may assume that the left premiss was derived by means of another rule. We suppose w.l.o.g. that the left premiss was derived by means of a unary rule as the binary case for $L\to$ is similar, meaning our Cut is of the following form:
\begin{center}\scalebox{1}{
\AxiomC{$\Se_{1}\{X_{1};\Gamma_{1} \Rightarrow \Delta_{1},x=y\}$}
\RightLabel{$\infrule R1$}
\UnaryInfC{$\Se\{X;\GSD,x=y\}$}

\AxiomC{$\Se_{2}\{X_{2};x=y,\Gamma_{2} \Rightarrow \Delta_{2}\}$}
\RightLabel{$\infrule R2$}
\UnaryInfC{$\Se\{X;x=y,\GSD\}$}

\RightLabel{$\infrule Cut$}
\BinaryInfC{$\Se\{X;\GSD\}$}
\DisplayProof}
\end{center}
 As shown below, we can resolve the case by applying the height-preserving invertibility of {\em R1} to the right premiss of {\em Cut}, applying {\em Cut} with the premiss of {\em R1}, and then applying {\em R1} after (note that {\em R1} is applicable after the Cut since $x=y$ is neither auxiliary nor principal in {\em R1} by the shape of the rules in $\mathsf{NQ.L}$).
\begin{center}\scalebox{0.95}{
\AxiomC{$\Se_{1}\{X_{1};\Gamma_{1} \Rightarrow \Delta_{1},x=y\}$}

\AxiomC{$\Se_{2}\{X_{2};x=y,\Gamma_{2} \Rightarrow \Delta_{2}\}$}
\RightLabel{$\infrule R2$}
\UnaryInfC{$\Se\{X;x=y,\GSD\}$}
\RightLabel{$\infrule Lemma~\ref{lem:invert}$}
\UnaryInfC{$\Se_{1}\{X_{1};x=y,\Gamma_{1} \Rightarrow \Delta_{1}\}$}

\RightLabel{$\infrule Cut$}
\BinaryInfC{$\Se_{1}\{X_{1};\Gamma_{1} \Rightarrow \Delta_{1}\}$}
\RightLabel{$\infrule R1$}
\UnaryInfC{$\Se\{X;\GSD\}$}
\DisplayProof}
\end{center}
 If we suppose now that $x = y$ is principal in the left premiss of {\em Cut}, then the left premiss must be an initial sequent of the form $\Se\{X, x=y,\GSD,x=y\}$. We have cases according to whether $x=y$ is principal or not in the right premiss.
 If it is principal then the right premiss is either (i) an initial sequent or  (ii) the conclusion of an instance of   a rule in $\{Repl,\,Repl_X,\,Rig\}$. In case (i) the conclusion of {\em Cut} is an intial sequent and in case (ii) the conclusion of {\em Cut} is identical to the conclusion of its right premiss, which is cut-free derivable.
Else, the Cut is  of  the form shown below, where two copies of $x=y$ must occur in the right premiss since the contexts must match in Cut.
 \begin{center}
\AxiomC{$\Se\{X;x=y,\GSD,x=y\}$}

\AxiomC{$\Se'\{X';x=y,x=y,\Gamma' \Rightarrow \Delta'\}$}
\RightLabel{$\infrule R2$}
\UnaryInfC{$\Se\{X;x=y,x=y,\GSD\}$}

\RightLabel{$\infrule Cut$}
\BinaryInfC{$\Se\{X;x=y,\GSD\}$}
\DisplayProof
\end{center}
Applying the height-preserving admissible rule CL to the right premiss of {\em Cut} 
gives the desired conclusion.
 
 Let us suppose now that the weight of the cut formula is greater than zero. We also assume that the cut formula is principal in both premisses of {\em Cut} and consider the interesting cases when $A \equiv \fa xB$ and $A \equiv \Box B$ as all other cases are standard, see \cite[Thm.~5]{B09}. If the cut formula $A \equiv \fa xB$ is principal in both premisses of {\em Cut}, then our Cut is of the following form: 
\begin{center}
\AxiomC{$\Se\{X,y,z;\GSD,B(y/x)\}$}
\RightLabel{$\infrule R\fa$}
\UnaryInfC{$\Se\{X,z;\GSD,\fa xB\}$}
\AxiomC{$\Se\{X,z;B(z/x),\fa xB,\GSD\}$}
\RightLabel{$\infrule L\fa$}
\UnaryInfC{$\Se\{X,z;\fa xB,\GSD\}$}
\RightLabel{$\infrule Cut$}
\BinaryInfC{$\Se\{X,z;\GSD\}$}
\DisplayProof
\end{center}
 We first shift the Cut upward by applying the height-preserving admissibility of IW to the left premiss of {\em Cut}, and then apply {\em Cut} with the premiss of $L\fa$ as shown below, thus reducing $h_{1}+h_{2}$.
\begin{center}
\AxiomC{$\Se\{X,y,z;\GSD,B(y/x)\}$}
\RightLabel{$\infrule R\fa$}
\UnaryInfC{$\Se\{X,z;\GSD,\fa xB\}$}
\RightLabel{$\infrule IW$}
\UnaryInfC{$\Se\{X,z;B(z/x),\GSD,\fa xB\}$}
\AxiomC{$\Se\{X,z;B(z/x),\fa xB,\GSD\}$}
\RightLabel{$\infrule Cut$}
\BinaryInfC{$\Se\{X,z;B(z/x),\GSD\}$}
\DisplayProof
\end{center}
Let us refer to the above proof as $\mathcal{D}$. We now reduce the weight of the cut formula by applying Cut as shown below, giving the desired conclusion.
\begin{center}
\AxiomC{$\Se\{X,y,z;\GSD,B(y/x)\}$}
\RightLabel{$\infrule (z/y)$}
\UnaryInfC{$\Se\{X,z,z;\GSD,B(z/x)\}$}
\RightLabel{$\infrule SC$}
\UnaryInfC{$\Se\{X,z;\GSD,B(z/x)\}$}
\AxiomC{$\mathcal{D}$}
\RightLabel{$\infrule Cut$}
\BinaryInfC{$\Se\{X,z;\GSD\}$}
\DisplayProof
\end{center}

 We now assume that the cut formula $A \equiv \Box B$ is principal in both premisses and we may assume w.l.o.g. that the cut is an instance of $\mathsf{L}$-Cut. We consider the case where the right premiss of $\mathsf{L}$-Cut is an instance of $R_T$ and the left premiss of $\mathsf{L}$-Cut is an instance of $R\Box$. The remaining cases are proven in a similar fashion. The trick is to use the height-preserving admissibility of the special structural rules (see Lemma~\ref{specialstruct}), namely, the $S_{T}$ rule. Our $\mathsf{L}$-Cut is of the following form:\medskip

\noindent \scalebox{0.78}{
\AxiomC{$\Se\{X;\GSD,[\emptyset; \, \Rightarrow B]\}\{Y_{i}; \Pi_{i} \To \Sigma_{i}\}^{n}_{i=1}$}
\RightLabel{$\infrule R\Box$}
\UnaryInfC{$\Se\{X;\GSD,\Box B\}\{Y_{i}; \Pi_{i} \To \Sigma_{i}\}^{n}_{i=1}$}
\AxiomC{$\Se\{X;\Box B,B,\GSD\}\{Y_{i}; \Box B, \Pi_{i} \To \Sigma_{i}\}^{n}_{i=1}$}
\RightLabel{$\infrule R_T$}
\UnaryInfC{$\Se\{X;\Box B,\GSD\}\{Y_{i}; \Box B, \Pi_{i} \To \Sigma_{i}\}^{n}_{i=1}$}
\RightLabel{$\infrule{\mathsf{L}\text{-}Cut}$}
\BinaryInfC{$\Se\{X;\GSD\}\{Y_{i}; \Pi_{i} \To \Sigma_{i}\}^{n}_{i=1}$}
\DisplayProof}\medskip

\noindent Let $\mathcal{D}_{1}$ and $\mathcal{D}_{2}$ denote the derivation of the left and right premiss of $\mathsf{L}$-Cut, respectively. To resolve the case, we first apply the height-preserving admissible rule IW to the conclusion of $\mathcal{D}_1$, yielding the derivation $\mathcal{D}_{3}$ shown below top. We then apply $\mathsf{L}$-Cut to the conclusion of $\mathcal{D}_3$ and the premiss of $\mathcal{D}_2$ (where $h_{1}+h_{2}$ is strictly smaller), giving the second derivation shown below, which we refer to as $\mathcal{D}_{4}$. Finally, as shown in the third derivation below, we can apply Cut to $B$ (which has a strictly smaller weight than $\Box B$), and derive the desired conclusion after applying a single application of the admissible rule $S_{T}$ to the left premiss.
\begin{center}

$\mathcal{D}_3 \left \{
\begin{array}{l}
\phantom{a}\\\noalign{\smallskip}
\phantom{A}\\\noalign{\smallskip}
\phantom{A}\\
\end{array} \right.  $
\AxiomC{$\Se\{X;\GSD,[\emptyset; \, \Rightarrow B]\}\{Y_{i}; \Pi_{i} \To \Sigma_{i}\}^{n}_{i=1}$}
\RightLabel{$\infrule R\Box$}
\UnaryInfC{$\Se\{X;\GSD,\Box B\}\{Y_{i}; \Pi_{i} \To \Sigma_{i}\}^{n}_{i=1}$}
\RightLabel{$\infrule IW$}
\UnaryInfC{$\Se\{X;B,\GSD,\Box B\}\{Y_{i}; \Pi_{i} \To \Sigma_{i}\}^{n}_{i=1}$}
\DisplayProof
\end{center}
\begin{center}
$\mathcal{D}_4 \left \{
\begin{array}{l}
\phantom{A}\\\noalign{\smallskip}
\phantom{A}\\
\end{array} \right.  $
\AxiomC{$\mathcal{D}_3$}
\AxiomC{$\Se\{X;\Box B,B,\GSD\}\{Y_{i}; \Box B, \Pi_{i} \To \Sigma_{i}\}^{n}_{i=1}$}
\RightLabel{$\infrule{\mathsf{L}\text{-}Cut}$}
\BinaryInfC{$\Se\{X;B,\GSD\}\{Y_{i}; \Pi_{i} \To \Sigma_{i}\}^{n}_{i=1}$}
\DisplayProof
\end{center}

\begin{center}
\AxiomC{$\Se\{X;\GSD,[\emptyset; \, \Rightarrow B]\}\{Y_{i}; \Pi_{i} \To \Sigma_{i}\}^{n}_{i=1}$}
\RightLabel{$\infrule S_{T}$}
\UnaryInfC{$\Se\{X;\GSD, B\}\{Y_{i}; \Pi_{i} \To \Sigma_{i}\}^{n}_{i=1}$}

\AxiomC{$\mathcal{D}_{4}$}

\RightLabel{$\infrule Cut$}
\BinaryInfC{$\Se\{X;\GSD\}\{Y_{i}; \Pi_{i} \To \Sigma_{i}\}^{n}_{i=1}$}
\DisplayProof
\end{center}
\qed\end{proof}

\section{Soundness and Completeness}\label{sec:characterisation}

\begin{theorem}[Soundness] If $\mathsf{NQ.L}\vdash \Se$ then $\texttt{fm}(\Se)$ is $\mathsf{Q.L}$-valid.
\end{theorem}

\begin{proof} We first note that nested application of rules is sound: for each context $\Se\{\cdot\}$, if $A\to B$ is $\mathsf{Q.L}$-valid then $\texttt{fm}(\Se\{A\}) \to \texttt{fm}(\Se\{B\})$ is $\mathsf{Q.L}$-valid. This can be shown by induction on the depth of the context $\Se\{\cdot\}$; see \cite[Lem.~3]{B09} for details.

The $\mathsf{Q.L}$-soundness of the rules of $\mathsf{NQ.L}$ is proved by induction on the height of the derivation. The cases of initial sequents and of  propositional rules of $\mathsf{NQ.L}$ are given in \cite[Thm.~1]{B09}. We present the cases of $L\fa$, $R_{cbf}$,  $Rig$, and $R_{5dom}$, all other cases being similar. If $\texttt{fm}(X,z; A(z/x),\fa xA,\GSD)$ is $\mathsf{Q.L}$-valid, then the $\mathsf{Q.L}$-validity of $\texttt{fm}(X,z; \fa xA,\GSD)$ follows by the soundness of the axiom {\bf UI}$^\E$. If $\texttt{fm}(X,x;\GSD,[Y,x;\Pi\To\Sigma])$ is $\mathsf{Q.L.CBF}$-valid, then the formula  $\texttt{fm}(X,x;\GSD,[Y;\Pi\To\Sigma])$ is as well because frames for $\mathsf{Q.L.CBF}$ have increasing domains. The $\mathsf{Q.L}$-validity of $\texttt{fm}(\Se\{X;x=y,\GSD\}\{Y;\Pi\To\Sigma\})$ follows from that of $\texttt{fm}(\Se\{X;x=y,\GSD\}\{Y;x=y,\Pi\To\Sigma\})$ since variables are rigid designators---i.e., the validity of $\mathbf{NI}:= x=y\to\Box(x=y)$ and that of $\mathbf{ND}$ allow identities to be duplicated up and down the accessibility relation, respectively. Finally, we argue that $R_{5dom}$ preserves $\mathsf{Q.L}$-validity when either $\mathbf{5},\mathbf{CBF}\in\mathsf{L}$ or  $\mathbf{5},\mathbf{BF}\in\mathsf{L}$. We show this holds for the  following one-context  rules  from which  $R_{5dom}$ is $\mathsf{NQ.L}$-derivable (if $x$ is in the signature of a non-root node, these rules bottom-up copy $x$ into the signature of another non-root node):
$$
\scalebox{0.95}{\infer[\infrule R_{5
dom_1}]{\Se\{[X,x;\GSD],[Y;\Pi \Rightarrow \Sigma]\}}{\Se\{[X,x;\GSD],[Y,x;\Pi \Rightarrow \Sigma]\}}
\quad
\infer[\infrule R_{5
dom_2}]{\Se\{[X,x;\GSD,[Y;\Pi \Rightarrow \Sigma]]\}}{\Se\{[X,x;\GSD,[Y,x;\Pi \Rightarrow \Sigma]]\}}
}
$$
$$
\infer[\infrule R_{5
dom_3}]{\Se\{[Y;\Pi \Rightarrow \Sigma,[X,x; \GSD]]\}}{\Se\{[Y,x;\Pi \Rightarrow \Sigma,[X,x;\GSD]]\}}
$$

\noindent If 
the premiss of one of these rule is $\mathsf{Q.L}$-valid,  then so is 
the respective conclusion since for {\bf 5}-frames with increasing  or decreasing domains the points satisfying $X,x;\Gamma \Rightarrow \Delta$ and $Y;\Pi \Rightarrow \Sigma$ are mutually accessible and have the same domain. 
\qed\end{proof}

\begin{theorem}[Completeness]   If $\texttt{fm}(\Se)$ is $\mathsf{Q.L}$-valid, then $\mathsf{NQ.L}\vdash \Se$.
\end{theorem}

\begin{proof} We show that $\mathsf{Q.L}\vdash \texttt{fm}(\Se)$ implies $\mathsf{NQ.L}\vdash \Se$; the theorem follows by the completeness of $\mathsf{Q.L}$ (Theorem~\ref{axcomp}). We proceed by induction on the height of the derivation of $\texttt{fm}(\Se)$ in $\mathsf{Q.L}$. 
The $\mathsf{NQ.L}$-admissibility of  rule {\bf MP}/{\bf UG}/{\bf N} is a corollary of Theorem~\ref{cut}/Lemma \ref{weakening}/Lemma \ref{merge}.   We consider only  axioms {\bf UI}$^\circ$ (assuming  $y\not\in A$ for simplicity), \textbf{ND}, and  {\bf CBF}. The cases of axioms {\bf REF} and {\bf REPL} follows from Lemma \ref{axid} and the other cases are similar.\medskip

\noindent\scalebox{0.74}{$$
\infer[\infrule R\fa]{\To \fa y(\fa xA\to A(y/x))}{
\infer[\infrule R\to]{y;\,\To \fa xA\to A(y/x)}{
\infer[\infrule L\fa]{y;\fa xA\To A(y/x)}{
\infer[\infrule L.~\ref{genax}]{y;A(y/x),\fa xA\To A(y/x)}{}}}}
\infer[\infrule R\Box]{x\neq y\To\Box (x\neq y)}{
\infer=[\infrule L\neg+R\neg]{x\neq y\To\, [\,\To x\neq y]}{
\infer[\infrule Rig]{\To x=y,[x=y\To\,]}{
\infer[\infrule L.~\ref{genax}]{x=y\To x=y,[x=y\To\,]}{}}}}\infer[\infrule R\fa]{\Box\fa xA\To\fa x\Box A}{
\infer[\infrule R\Box]{y;\Box\fa xA\To\Box A(y/x)}{
\infer[\infrule L\Box]{y;\Box\fa xA\To\,[\To A(y/x)]}{
\infer[\infrule R_{cbf}]{y;\Box\fa xA\To\,[\fa xA\To A(y/x)]}{
\infer[\infrule L\fa]{y;\Box\fa xA\To\,[y;\fa xA\To A(y/x)]}{
\infer[\infrule L.~\ref{genax}]{y;\Box\fa xA\To\,,[y;A(y/x),\fa xA\To A(y/x)]}{}}}}}}
$$}

\phantom{a}\qed\end{proof}

\section{Conclusion and Future Work}\label{sec:future}

 We provided a uniform nested sequent presentation of quantified modal logics characterised by combinations of fundamental properties. Due to the inclusion of equality in the language of the QMLs considered, our nested calculi permit a formula translation by means of the (definable) existence predicate. As a consequence, our systems possess both a good degree of modularity \emph{and} utilise a language as expressive as that of each logic, 
 yielding more economical systems in contrast to the labelled calculi given for the same QMLs, which employ a more expressive language~\cite{NP11,V00}. Beyond formula interpretability, our nested calculi satisfy fundamental properties such as the admissibility of important structural rules, invertibility of all rules, and syntactic cut-elimination.
 
 In future work, we aim to investigate constructive proofs of interpolation properties with our nested calculi (cf.~\cite{FK15,L20b}), 
 to use (variations of) our nested calculi to identify decidable QML fragments, as well as extend the present approach to QMLs with non-rigid designators and, possibly, definite descriptions based on $\lambda$-abstraction (see \cite{FM98}) as was done in \cite{O21} for labelled sequent calculi. Another open problem is to give nested sequents with a formula interpretation for QMLs where the existence predicate is not expressible; we conjecture that this might be achieved by using the  `universally closed nesting' defined by \mbox{Br\"unner  for 
 free logics   \cite{B10}.}
 
 We also aim to generalise our approach by employing a wider selection of propagation rules~\cite{CCGH97,F72} and reachability rules~\cite{L21,L22} in our systems. As shown in various works~\cite{GPT11,L21}, diverse classes of logics characterised by Horn properties can be supplied cut-free nested calculi by utilising logical rules that propagate or consume data along paths within nested sequents specified by formal grammars. Applying this technique, we plan to see if we can capture a much wider class of QMLs in a uniform and modular fashion, and plan to investigate admissibility and invertibility properties as well as cut-elimination in this more general setting. It would also be worthwhile to examine the relationship between our nested calculi and other calculi for QMLs; e.g., we could study the computational relationship between our nested calculi and the labelled calculi for QMLs, showing how proofs can be translated and determining complexity bounds for the relative sizes of proofs.

%
%
%
%

\appendix

\section{Admissibility of Special Structural Rules}





\begin{lemma}\label{lem:T-admiss} If $\mathsf{NQ.L}$ contains the rule $R_T$ for the propositional axiom ${\bf T}$, then the corresponding structural rule $S_T$ is admissible in $\mathsf{NQ.L}$.
\end{lemma}

\begin{proof} By induction on the height of the given derivation. The base case is trivial, so we focus on the inductive step. We consider the interesting cases of $L\Box$, $Rig$, $R_{bf}$, $R_5$ and $R_{5dom}$ and note that the remaining cases are simpler or similar. The non-trivial $L\Box$ case is shown below left and can be resolved as shown below right by applying the inductive hypothesis and then $R_{T}$. 
\begin{center}
\resizebox{\columnwidth}{!}{
\begin{tabular}{cc c}
\AxiomC{$\Se\{X;\Box A,\GSD,[Y;A,\Pi\To\Sigma]\}$}
\RightLabel{$\infrule L\Box$}
\UnaryInfC{$\Se\{X;\Box A,\GSD,[Y;\Pi\To\Sigma]\}$}
\RightLabel{$\infrule S_{T}$}
\UnaryInfC{$\Se\{X,Y;\Box A,\Pi,\GSD,\Sigma\}$}
\DisplayProof

&$\leadsto$&

\AxiomC{$\Se\{X;\Box A,\GSD,[Y;A,\Pi\To\Sigma]\}$}
\RightLabel{$\infrule IH$}
\UnaryInfC{$\Se\{X;\Box A,A,\Pi,\GSD,\Sigma\}$}
\RightLabel{$\infrule R_{T}$}
\UnaryInfC{$\Se\{X,Y;\Box A,\Pi,\GSD,\Sigma\}$}
\DisplayProof
\end{tabular}
}
\end{center}
 A non-trivial $Rig$ case is shown below:
\begin{center}
\resizebox{\columnwidth}{!}{
\begin{tabular}{ccc}
\AxiomC{$\Se\{X;x=y,\GSD,[Y;x=y,\Pi\To\Sigma]\}$}
\RightLabel{$\infrule Rig$}
\UnaryInfC{$\Se\{X;x=y,\GSD,[Y;\Pi\To\Sigma]\}$}
\RightLabel{$\infrule S_{T}$}
\UnaryInfC{$\Se\{X,Y;x=y,\Pi,\GSD,\Sigma\}$}
\DisplayProof 
&$\leadsto$&\AxiomC{$\Se\{X;x=y,\GSD,[Y;x=y,\Pi\To\Sigma]\}$}
\RightLabel{$\infrule IH$}
\UnaryInfC{$\Se\{X,Y;x=y,x=y,\Pi,\GSD,\Sigma\}$}
\RightLabel{$\infrule CL$}
\UnaryInfC{$\Se\{X,Y;x=y,\Pi,\GSD,\Sigma\}$}
\DisplayProof
\end{tabular} }
\end{center}
The non-trivial $R_{bf}$ case is similar to the $Rig$ case above and is shown below.
\begin{center}
\resizebox{\columnwidth}{!}{
\begin{tabular}{ccc}
\AxiomC{$\Se\{X,x;\GSD,[Y,x;\Pi\To\Sigma]\}$}
\RightLabel{$\infrule R_{bf}$}
\UnaryInfC{$\Se\{X;\GSD,[Y,x;\Pi\To\Sigma]\}$}
\RightLabel{$\infrule S_{T}$}
\UnaryInfC{$\Se\{X,Y,x;\Pi,\GSD,\Sigma\}$}
\DisplayProof

&$\leadsto$&

\AxiomC{$\Se\{X,x;\GSD,[Y,x;\Pi\To\Sigma]\}$}
\RightLabel{$\infrule IH$}
\UnaryInfC{$\Se\{X,Y,x,x;\Pi,\GSD,\Sigma\}$}
\RightLabel{$\infrule SC$}
\UnaryInfC{$\Se\{X,Y,x;\Pi,\GSD,\Sigma\}$}
\DisplayProof
\end{tabular}
}
\end{center}
We consider an interesting $R_5$ case:
\begin{center}
\AxiomC{$Z;\Pi_{1}\To\Sigma_{1},[X;\Box A,\GSD],[\Se\{Y;\Box A,\Pi_{2}\To\Sigma_{2}\}]$}
\RightLabel{$\infrule R_5$}
\UnaryInfC{$Z;\Pi_{1}\To\Sigma_{1},[X;\Box A,\GSD],[\Se\{Y;\Pi_{2}\To\Sigma_{2}\}]$}
\RightLabel{$\infrule S_{T}$}
\UnaryInfC{$Z,X;\Box A,\Pi_{1},\Gamma\To\Delta,\Sigma_{1},[\Se\{Y;\Pi_{2}\To\Sigma_{2}\}]$}
\DisplayProof
\end{center}
The above case can be resolved as shown below by first applying the inductive hypothesis to the derivation of the premiss of $R_5$. After, we apply IW to weaken in $\Box A$ along the path from the root to $Y;\Box A, \Pi_{2}\To\Sigma_{1}$ (say, $n-1$ times), and then successively apply $R_4$ ($n$ times) to obtain the desired conclusion. Note that $R_4$ is applicable since we have assumed all of our nested calculi properly closed, and $R_4$ must be a rule if both $R_T$ and $R_5$ are assumed to be rules of our calculus.
\begin{center}
\AxiomC{$Z;\Pi_{1}\To\Sigma_{1},[X;\Box A,\GSD],[\Se\{Y;\Box A,\Pi_{2}\To\Sigma_{2}\}]$}
\RightLabel{$\infrule IH$}
\UnaryInfC{$Z,X;\Box A,\Pi_{1},\Gamma\To\Delta,\Sigma_{1},[\Se\{Y;\Box A, \Pi_{2}\To\Sigma_{2}\}]$}
\RightLabel{$\infrule{ IW\, \text{ ($n-1$ times)}}$}
\UnaryInfC{$Z,X;\Box A,\Pi_{1},\Gamma\To\Delta,\Sigma_{1},[\Se\{Y;\Box A, \Pi_{2}\To\Sigma_{2}\}]$}
\RightLabel{$\infrule{ R_4\, \text{ ($n$ times)}}$}
\UnaryInfC{$Z,X;\Box A,\Pi_{1},\Gamma\To\Delta,\Sigma_{1},[\Se\{Y;\Pi_{2}\To\Sigma_{2}\}]$}
\DisplayProof
\end{center}
Last, we consider an interesting $R_{5dom}$ case:
\begin{center}
\AxiomC{$Z;\Pi_{1}\To\Sigma_{1},[X,x;\GSD],[\Se\{Y,x;\Pi_{2}\To\Sigma_{2}\}]$}
\RightLabel{$\infrule R_{5dom}$}
\UnaryInfC{$Z;\Pi_{1}\To\Sigma_{1},[X;\GSD],[\Se\{Y,x;\Pi_{2}\To\Sigma_{2}\}]$}
\RightLabel{$\infrule S_{T}$}
\UnaryInfC{$Z,X;\Pi_{1},\Gamma\To\Delta,\Sigma_{1},[\Se\{Y,x;\Pi_{2}\To\Sigma_{2}\}]$}
\DisplayProof
\end{center}
 We apply the inductive hypothesis, followed by the height-preserving admissible rule SW, which we apply $n-1$ times adding the variable $x$ along the path from the root to $Y,x;\Pi_{2}\To\Sigma_{2}$. Finally, we apply the $R_{bf}$ rule $n$ times to delete the $n-1$ copies of $x$ to the root, thus giving the desired conclusion. We note that we may apply $R_{bf}$ as our nested calculi are properly closed, meaning that ${\bf BF} \in \mathsf{L}$ since this case assumes that ${\bf T}, {\bf 5}, {\bf CBF} \in \mathsf{L}$.
\begin{center}
\AxiomC{$Z;\Pi_{1}\To\Sigma_{1},[X,x;\GSD],[\Se\{Y,x;\Pi_{2}\To\Sigma_{1}\}]$}
\RightLabel{$\infrule IH$}
\UnaryInfC{$Z,X,x;\Pi_{1},\Gamma\To\Delta,\Sigma_{1},[\Se\{Y,x; \Pi_{2}\To\Sigma_{2}\}]$}
\RightLabel{$\infrule{ SW\, \text{ ($n-1$ times)}}$}
\UnaryInfC{$Z,X,x;\Pi_{1},\Gamma\To\Delta,\Sigma_{1},[\Se\{Y,x; \Pi_{2}\To\Sigma_{2}\}]$}
\RightLabel{$\infrule{ R_{bf}\, \text{ ($n$ times)}}$}
\UnaryInfC{$Z,X;\Pi_{1},\Gamma\To\Delta,\Sigma_{1},[\Se\{Y,x;\Pi_{2}\To\Sigma_{2}\}]$}
\DisplayProof
\end{center}
\qed\end{proof}

\begin{lemma}\label{lem:4-admiss} If $\mathsf{NQ.L}$ contains the rule $R_4$ for the propositional axiom ${\bf 4}$, then the corresponding structural rule $S_4$ is admissible in $\mathsf{NQ.L}$.
\end{lemma}

\begin{proof} By induction on the height of the given derivation. As the base case and most cases of the inductive step are trivial, we only consider the interesting $R_4$ and $R_{cbf}$ cases, noting that all others are straightforward or similar. We first show how to resolve the interesting $R_4$ case below:
\begin{center}
\resizebox{\columnwidth}{!}{
\begin{tabular}{ccc}
\AxiomC{$\Se\{X; \Box A, \GSD, [Y; \Box A, \Pi \To \Sigma]\}$}
\RightLabel{$\infrule{R_4}$}
\UnaryInfC{$\Se\{X; \Box A, \GSD, [Y; \Pi \To \Sigma]\}$}
\RightLabel{$\infrule{S_4}$}
\UnaryInfC{$\Se\{X; \Box A, \GSD, [\emptyset; \, \To \, [Y; \Pi \To \Sigma]]\}$}
\DisplayProof
&$\leadsto$&
\AxiomC{$\Se\{X; \Box A, \GSD, [Y; \Box A, \Pi \To \Sigma]\}$}
\RightLabel{$\infrule{IH}$}
\UnaryInfC{$\Se\{X; \Box A, \GSD, [\emptyset; \, \To \, [Y; \Box A, \Pi \To \Sigma]]\}$}
\RightLabel{$\infrule{IW}$}
\UnaryInfC{$\Se\{X; \Box A, \GSD, [\emptyset; \Box A \To \, [Y; \Box A, \Pi \To \Sigma]]\}$}
\RightLabel{$\infrule{R_4}$}
\UnaryInfC{$\Se\{X; \Box A, \GSD, [\emptyset; \Box A \To \, [Y; \Pi \To \Sigma]]\}$}
\RightLabel{$\infrule{R_4}$}
\UnaryInfC{$\Se\{X; \Box A, \GSD, [\emptyset; \, \To \, [Y; \Pi \To \Sigma]]\}$}
\DisplayProof
\end{tabular}}
\end{center}
 The non-trivial $R_{cbf}$ case is shown below:
\begin{center}
\AxiomC{$\Se\{X,x; \GSD, [Y,x; \Box A, \Pi \To \Sigma]\}$}
\RightLabel{$\infrule{R_{cbf}}$}
\UnaryInfC{$\Se\{X,x; \GSD, [Y; \Box A, \Pi \To \Sigma]\}$}
\RightLabel{$\infrule{S_4}$}
\UnaryInfC{$\Se\{X,x; \GSD, [\emptyset; \, \To \, [Y; \Box A, \Pi \To \Sigma]]\}$}
\DisplayProof
\end{center}
 We resolve the case by applying IH, SW, and then two instances of $R_{cbf}$.
\begin{center}
\AxiomC{$\Se\{X,x; \GSD, [Y,x; \Box A, \Pi \To \Sigma]\}$}
\RightLabel{$\infrule{IH}$}
\UnaryInfC{$\Se\{X,x; \GSD, [\emptyset; \, \To \, [Y,x; \Box A, \Pi \To \Sigma]]\}$}
\RightLabel{$\infrule{SW}$}
\UnaryInfC{$\Se\{X,x; \GSD, [x; \, \To \, [Y,x; \Box A, \Pi \To \Sigma]]\}$}
\RightLabel{$\infrule{R_{cbf}\, \text{ (2 times)}}$}
\UnaryInfC{$\Se\{X,x; \GSD, [\emptyset; \, \To \, [Y; \Box A, \Pi \To \Sigma]]\}$}
\DisplayProof
\end{center}
\qed\end{proof}

\begin{lemma}\label{lem:L-Str-admiss} The rule $\mathsf{L}$-Str is admissible in $\mathsf{NQ.L}$.
\end{lemma}

\begin{proof} By induction on the height of the given derivation. The base case is trivial. We show the $R_{cbf}$ case of the inductive step under the assumption that $\mathbf{5} \in \mathsf{L}$ and $\mathbf{4} \not\in \mathsf{L}$. The remaining cases are straightforward or similar.
\begin{center}
\AxiomC{$\Se\{X_{1},x; \Gamma_{1} \Rightarrow \Delta_{1}, [Y,x; \Pi \To \Sigma]\}\{X_{2}; \Gamma_{2} \Rightarrow \Delta_{2}\}$}
\RightLabel{$\infrule{R_{cbf}}$}
\UnaryInfC{$\Se\{X_{1},x; \Gamma_{1} \Rightarrow \Delta_{1}, [Y; \Pi \To \Sigma]\}\{X_{2}; \Gamma_{2} \Rightarrow \Delta_{2}\}$}
\RightLabel{$\infrule{\mathsf{L}\text{-Str}}$}
\UnaryInfC{$\Se\{X_{1},x; \Gamma_{1} \Rightarrow \Delta_{1}\}\{X_{2}; \Gamma_{2} \Rightarrow \Delta_{2},[Y; \Pi \To \Sigma]\}$}
\DisplayProof
\end{center}
 Since $\mathsf{NQ.L}$ is assumed to be properly closed, we know that $R_{5dom}$ is a rule in the calculus as $\mathbf{5}, \mathbf{CBF} \in \mathsf{L}$. 
 The case above can be resolved by applying the inductive hypothesis followed by $R_{5dom}$. 
\qed\end{proof}

\begin{lemma}\label{lem:5-admiss} If $\mathsf{NQ.L}$ contains the rule $R_5$ for the propositional axiom ${\bf 5}$, then the corresponding structural rule $S_5$ is admissible in $\mathsf{NQ.L}$.
\end{lemma}

\begin{proof} Immediate from Lemma~\ref{lem:L-Str-admiss} as $S_5$ is an instance of $\mathsf{L}$-Str if ${\bf 5} \in \mathsf{L}$.
\qed\end{proof}

\end{document}